\documentclass[twoside, 12pt]{article}
\usepackage{a4wide, amsthm, amssymb, amscd, mathptmx}

\usepackage{amsfonts}

\usepackage[all]{xy}\CompileMatrices\SelectTips{cm}{12}

\theoremstyle{plain}
\newtheorem{Thm}{\sc Theorem}[section]
\newtheorem{Theorem}[Thm]{\sc Theorem}
\newtheorem{Corollary}[Thm]{\sc Corollary}
\newtheorem*{Corollary*}{\sc Corollary}
\newtheorem{Proposition}[Thm]{\sc Proposition}
\newtheorem*{Proposition*}{\sc Proposition}
\newtheorem{Lemma}[Thm]{\sc Lemma}

\theoremstyle{definition}
\newtheorem{Definition}[Thm]{Definition}

\theoremstyle{remark}
\newtheorem{Remark}[Thm]{Remark}

\newtheorem*{Example*}{Example}
\newtheorem*{Remark*}{Remark}

\newcommand{\et}{\mathop{\rm \acute{e}t}}

\newcommand{\id}{{\mathop{\rm id}}}
\newcommand{\ZZ}{{\mathbb Z}}
\newcommand{\FF}{{\mathbb F}}
\newcommand{\Spec}{\mathop{\rm Spec}}
\newcommand{\Diag}{\mathop{\rm Diag}}
\renewcommand{\O}{{\cal O}}

\newcommand{\E}{{\cal E}}

\newcommand{\N}{{\cal N}}

\newcommand{\cQ}{{\cal Q}}

\newcommand{\NN}{{\mathbb N}}
\newcommand{\GG}{{\mathbb G}}

\newcommand{\PP}{{\mathbb P}}

\newcommand{\C}{{\cal C}}
\newcommand{\ch}{\mathop{\rm ch}}

\newcommand{\td}{\mathop{\rm td}}

\newcommand{\ti}{\tilde}
\newcommand{\Hom}{{\mathop{{\rm Hom}}}}
\newcommand{\cHom}{{\mathop{{\cal H}om}}}
\newcommand{\End}{{\mathop{ End}}}
\newcommand{\Vect}{{\mathop{{\cal C}^{\rm nf}}}}

\newcommand{\num}{\mathop{\rm num}}

\newcommand{\mod}{\mathop{\rm mod}}
\newcommand{\GL}{\mathop{\rm GL}}

\newcommand{\Pic}{\mathop{\rm Pic}}
\newcommand{\rk}{\mathop{{rk}}}
\newcommand{\alb}{\mathop{\rm alb}}
\newcommand{\Alb}{\mathop{\rm Alb}}
\newcommand{\NS}{\mathop{\rm NS}}

\mathchardef\mhyp="2D

\begin{document}

\markboth{\rm A.\ Langer} {\rm On the S-fundamental group scheme
II}

\title{On the S-fundamental group scheme II
}
\author{Adrian Langer}
\date{\today}

\maketitle


{\sc Address:}\\
1. Institute of Mathematics, Warsaw University,
ul.\ Banacha 2, 02-097 Warszawa, Poland,\\
2. Institute of Mathematics, Polish Academy of Sciences,
ul.~\'Sniadeckich 8, 00-956 Warszawa, Poland.\\
e-mail: {\tt alan@mimuw.edu.pl}

\begin{abstract}
The S-fundamental group scheme is the group scheme corresponding
to the Tannaka category of numerically flat vector bundles. We use
determinant line bundles  to prove that the S-fundamental group of
a product of two complete varieties is a product of their
S-fundamental groups as conjectured by  V. Mehta and the author.
We also compute the abelian part of the S-fundamental group scheme
and the S-fundamental group scheme of an abelian variety or a
variety with trivial \'etale fundamental group.
\end{abstract}

\section*{Introduction}

Let $X$ be a complete, reduced and connected
scheme defined over an algebraically closed field $k$. A vector
bundle $E$ on $X$ is called \emph{numerically flat} if both $E$
and its dual $E^*$ are nef vector bundles.

The category $\Vect (X)$ of numerically flat vector bundles is a
$k$-linear abelian rigid tensor category. Fixing a closed point

$x\in X$ endows $\Vect (X)$ with a fiber functor $E\to E(x)$ which
makes $\Vect (X)$ into a neutral Tannaka category. Hence there is
an equivalence between $\Vect (X)$ and the category of
representations of some affine group scheme $\pi ^S_1(X,x)$ that
we call an \emph{S-fundamental group scheme} of $X$ with base
point $x$. The S-fundamental group scheme appeared in the curve case
in \cite{BPS} and then in general in \cite{La} and \cite{Me}.

The strategy of constructing a similar fundamental group
scheme using another Tannaka category of essentially finite vector bundles
goes back to Nori's influential paper \cite{No}. In this paper Nori defined a smaller group
scheme, called \emph{Nori's fundamental group scheme}, which is a pro-finite completion
of the S-fundamental group scheme. Therefore the S-fundamental group scheme plays with
respect to Nori's fundamental group scheme a similar role as the fundamental group scheme
for stratified sheaves with respect to the \'etale fundamental group scheme.

In \cite{La} we exploited another interpretation of the
S-fundamental group scheme in case of smooth projective varieties.
Namely, numerically flat vector bundles are precisely strongly
semistable torsion free sheaves with vanishing Chern classes.
Using it one can apply vanishing theorems to establish, e.g.,
Lefschetz type theorems for the S-fundamental group scheme (see
\cite[Theorems 10.2, 10.4 and 11.3]{La}).

This paper further exploits an observation that all the known
properties of Nori's fundamental group scheme should still be valid for
the more general S-fundamental group scheme. This allows in particular
to obtain interesting properties of numerically flat bundles.

As a main result of this paper we prove that the S-fundamental
group of a product of two complete varieties is a product of their
S-fundamental groups. This result was conjectured both in
\cite{La} and \cite[Remark 5.12]{Me}. It implies in particular the
corresponding result for Nori's fundamental group that was
conjectured by Nori in \cite{No} and proven by  Mehta and
Subramanian in \cite{MS}. Our proof of Nori's conjecture is
completely different from that in \cite{MS}. On the other hand, our
proof gives also the corresponding result for the \'etale
fundamental group scheme that was used in the proof of Mehta and
Subramanian.

As a corollary to our theorem we prove that the reduced scheme
underlying the torsion component of the identity of the Picard
scheme of a product of projective varieties is a product of the
corresponding torsion components of its factors (see Corollary
\ref{Pic-product}). The author does not know any other proof of
this fact.

The proof of the main theorem follows easily from the fact that
the push forward of a numerically flat sheaf on the product
$X\times Y$ to $X$ is also numerically flat (and in particular
locally free). To prove this result we proceed by induction. In
the curve case we employ some determinants of line bundles.

In the remaining part of the paper we compute the abelian part of
the S-fundamental group scheme (cf. \cite[Lemma 20 and Theorem
21]{dS} for a similar result for the fundamental group scheme for
stratified sheaves). This allows us to compute the S-fundamental
group scheme for abelian varieties (this can also be done using
an earlier result of Mehta and Nori in \cite{MN} for which we again
have a completely different proof).

The last part of the paper is based on \cite{EM1} and \cite{EM2}
and it contains computation of the S-fundamental group for
varieties with trivial \'etale fundamental group.

\medskip

\section{Preliminaries}

In this section we gather a few auxiliary results.

\subsection{Numerical equivalence}

Let $X$ be a complete, reduced and connected $d$-dimensional
scheme defined over an algebraically closed field $k$.
We say that a rank $r$  locally free sheaf $E$ on $X$ is \emph{numerically trivial} if
for every coherent sheaf $F$ on $X$ we have $\chi (X, F\otimes E)= r\chi (X, F)$.

Now assume that $X$ is smooth. Then we define the \emph{numerical
Grothendieck group} $K(X)_{\num}$ as the Grothendieck group (ring)
$K(X)$ of coherent sheaves modulo numerical equivalence, i.e.,
modulo the radical of the quadratic form given by the Euler
characteristic $(a,b)\mapsto \chi (a\cdot b)=\int _X \ch (a)\ch (b)\td
(X)$.

We say that a coherent sheaf  has \emph{numerically trivial Chern
classes} if there exists an integer $r$ such that the class
$[E]-r[\O_X]$ is zero in $K(X)_{\num}$. By the Riemann--Roch
theorem this is equivalent to the vanishing of numerical Chern
classes.

\medskip

\subsection{Nefness}

Let us recall that a locally free sheaf $E$ on a complete
$k$-scheme $X$ is called \emph{nef} if and only if $\O_ {\PP
(E)}(1)$ is nef on the projectivization $\PP (E)$ of $E$. A
locally free sheaf $E$ is nef if and only if for any  morphism $f:
C\to X$ from a smooth projective curve $C$ each quotient of $f^*E$
has a non-negative degree.

We say that $E$ is \emph{numerically flat} if both $E$ and $E^*$
are nef. A locally free sheaf $E$ is numerically flat if and only
if for any morphism $f: C\to X$ from a smooth projective curve $C$
the pull-back $f^*E$ is semistable of degree zero.

If $X$ is projective then any numerically flat sheaf is
numerically trivial and a line bundle is numerically flat if and only
if it is numerically trivial.

\medskip

\subsection{Picard schemes}

Let $X$ be an (integral) variety defined over an algebraically
closed field $k$. Let $\Pic X$ denote the Picard group scheme of
$X$. By $\Pic ^0 X$ we denote its connected component of the
identity and by $\Pic ^{\tau} X$ we denote the torsion component
of the identity. Note that all these group schemes can be
non-reduced.

If $X$ is projective then $\Pic ^{\tau} X$ is of finite type and
the torsion group $\Pic ^{\tau} X/ \Pic^0 X$ is finite. The scheme
$\Pic ^{\tau} X$ represents the functor of numerically trivial
line bundles. If $X$ is smooth and projective then $\Pic ^{\tau} X
$ is the fine moduli space of torsion free rank $1$ sheaves with
numerically trivial Chern classes (e.g., because by Theorem
\ref{num-nef} every torsion free rank $1$ sheaf with numerically
trivial Chern classes is in fact a line bundle).

\medskip

\subsection{Cohomology and base change}\label{cohomology_and_base-change}

Let $f:X\to Y$ be a proper morphism of noetherian schemes and let
$E$ be a $Y$-flat coherent $\O_X$-module. Then we say that
\emph{cohomology and base change commute for $E$ in degree $i$} if
for every base change diagram
$$\xymatrix{
& X\times_Y Y'\ar[d]^{g} {\ar[r]^{v}} & X\ar[d]^f\\
& Y'\ar[r]^{u}&  Y}$$ the natural map $u^*R^if_*E\to R^ig_*(v^*E)$
is an isomorphism.

If for every point $y\in Y$ the natural map
$R^if_*E\otimes k(y)\to H^i (X_y, E_y)$ is an isomorphism then
cohomology and base change commute for $E$ in degree $i$.
Indeed, our assumption implies that $u^*R^if_*E\to R^ig_*(v^*E)$ is
an isomorphism at every point of $Y'$.

\subsection{Boundedness of numerically flat sheaves}

The following theorem is a corollary of the general boundedness
result \cite[Theorem 4.4]{La1}.
\medskip

\begin{Theorem} \label{bound}
The set of numerically flat sheaves on a $d$-dimensional normal
projective variety $X$ is bounded.
\end{Theorem}

\begin{proof}
Let $E$ be a rank $r$ numerically flat sheaf on $X$. Let us fix a
very ample line bundle $\O_X(1)$ on $X$. Then we can write the
Hilbert polynomial of $E$ as
$$P(E)(m)=\chi (X, E(m))= \sum _{i=0}^d \chi (E|_{{\bigcap}_{j\le i}H_j}) {{m+i-1}\choose i},$$
where $H_1,...,H_d\in |\O_X(1)|$ are general hyperplane sections.
First let us note that $\chi (E|_{{\bigcap}_{j\le i}H_j})=r \chi
(\O_{{\bigcap}_{j\le i}H_j})$ for $i=d,d-1,d-2$. In cases $i=d$ or
$d-1$ the assertion is clear since $E|_{{\bigcap}_{j\le i}H_j}$ is
a numerically flat vector bundle on a set of points or on a smooth
curve, so the assertion follows from the usual Riemann--Roch
theorem.

In the case when $i=d-2$ it is sufficient to show that a rank $r$
numerically flat sheaf $F$ on a normal projective surface $Y$
satisfies equality $\chi (Y,F)=r\chi (Y, \O_Y)$. To prove this let
us take a resolution of singularities $f:\ti Y\to Y$. Then $\chi
(\ti Y,f^*F)=r\chi (\ti Y, \O_{\ti Y})$, because $f^*F$ is a
vector bundle with trivial Chern classes (as follows, e.g., from
Theorem \ref{num-nef}). But using the Leray spectral sequence
$H^j(\ti Y, R^if_*F)\Rightarrow H^{i+j}(Y, F)$ and the projection
formula we see that
$$\chi (\ti Y,f^*F)= \chi (Y, F)+\chi (Y, F\otimes R^1f_*\O_{\ti Y}).$$
Similarly, we have
$$\chi (\ti Y,\O_{\ti Y})= \chi (Y, \O_Y)+\chi (Y, R^1f_*\O_{\ti Y}).$$
Since $R^1f_*\O_{\ti Y}$ is supported on a finite number of points
and $F$ is locally free of rank $r$ we see that $\chi (Y, F\otimes
R^1f_*\O_{\ti Y})= r \chi (Y, R^1f_*\O_{\ti Y})$, which proves the
required equality.

Now let us write
$$P(E)(m)= \sum _{i=0}^d a_i {{m+d-i}\choose d-i}.$$
Then our assertion implies that $a_0(E), a_1(E)$ and $a_2(E)$
depend only on $X$. But the set of reflexive semistable sheaves
with fixed $a_0(E), a_1(E)$ and $a_2(E)$ on a normal projective
variety is bounded by \cite[Theorem 4.4]{La1}.
\end{proof}

\medskip
In case of smooth varieties the above theorem is an immediate
corollary of Theorem \ref{num-nef} and \cite[Theorem 4.4]{La1}.

\section{Fundamental groups in positive characteristic}

Let $X$ be a complete connected reduced scheme defined over an
algebraically closed field $k$.

\subsection{S-fundamental group scheme}

Let $\Vect (X)$ denote the full subcategory of the category of
coherent sheaves on $X$, which as objects contains all numerically
flat (in particular locally free) sheaves. Let us fix a $k$-point
$x\in X$. Then we can define the fiber functor $T_x: \Vect (X)\to
k\mhyp \mod$ by sending $E$ to its fiber $E(x)$. One can show that
$(\Vect (X), \otimes ,T_x, \O_X)$ is a neutral Tannaka category
(see \cite[Section 6]{La}). Therefore by \cite[Theorem 2.11]{DM}
the following definition makes sense:

\begin{Definition} \label{Def}
The affine $k$-group scheme  Tannaka dual to this neutral Tannaka
category is denoted by $\pi^S_1(X,x)$ and it is called the
\emph{S-fundamental group scheme} of $X$ with base point $x$.
\end{Definition}

This group scheme was first defined in the curve case by Biswas,
Parameswaran and Subramanian in \cite[Section 5]{BPS}, and then
independently in \cite{La} and \cite{Me}.

The following characterization of numerically flat bundles as
semistable sheaves with vanishing Chern classes appears in
\cite[Theorem 4.1 and Proposition 5.1]{La}.

\begin{Theorem} \label{num-nef}
Let $X$ be a smooth projective $k$-variety of dimension $d$. Let
$H$ be an ample divisor on $X$ and let $E$ be a coherent sheaf on
$X$. Then the following conditions are equivalent:
\begin{enumerate}
\item $E$ is a strongly $H$-semistable torsion free sheaf and its
Hilbert polynomial is the same as that of the trivial sheaf of the
same rank.
\item $E$ is a strongly $H$-semistable torsion free
sheaf and it has numerically trivial Chern classes.
\item $E$ is a strongly $H$-semistable reflexive sheaf with
$\ch _1(E)\cdot H^{d-1}=0$ and $\ch _2 (E)\cdot H^{d-2}=0$.
\item $E$ is locally free, nef and $c_1(E)H^{d-1}=0$.
\item $E$ is numerically flat.
\end{enumerate}
\end{Theorem}

\subsection{Nori's and \'etale fundamental group schemes}

Let us consider the category $\C ^N(X)$ of bundles which are
trivializable over a principal bundle under a finite group scheme.
For a $k$-point $x\in X$  we can define the fiber functor $T_x: \C
^N (X)\to k\mhyp \mod$ by sending $E$ to its fiber $E(x)$. This makes
$\C ^N (X)$ a neutral Tannaka category which is equivalent to the
category of representations of an affine group scheme $\pi
^N_1(X,x)$ called \emph{Nori's fundamental group scheme}.

If instead of $\C ^N (X)$ we consider the category $\C ^{\et}(X)$
of bundles which are trivializable over a principal bundle under a
finite \'etale group scheme then we get an \emph{\'etale
fundamental group} $\pi _1^ {\et} (X,x)$.

Note that both these group schemes can be recovered from $\pi
^S_1(X,x)$ as inverse limits of some directed systems (see, e.g.,
\cite[Section 6]{La}). In particular, any theorem proved for the
S-fundamental group scheme implies the corresponding theorems for
\'etale and Nori's fundamental group schemes.

\subsection{The unipotent part of the S-fundamental group scheme}
\label{unipotent}

The largest  unipotent quotient of the S-fundamental group scheme
of $X$ is called the \emph{unipotent part} of $\pi ^S(X,x)$ and
denoted by $\pi ^U(X,x)$. By the standard Tannakian considerations
we know that the category of finite dimensional
$k$-representations of $\pi ^U(X,x)$ is equivalent (as a neutral
Tannakian category) to the category $\N\Vect (X)$ defined as the
full subcategory of $\Vect (X)$ whose objects have a filtration
with all quotients isomorphic to $\O_X$.

Nori proves that in positive characteristic $\pi ^U(X,x)$ is a
pro-finite group scheme (see \cite[Chapter IV, Proposition 3]{No})
and hence $\pi ^U(X,x)$ is the unipotent part of $\pi ^N (X,x)$.
This is no longer true in the characteristic zero case. However,
in arbitrary characteristic we know that $\Hom (\pi ^U (X,x), \GG
_a)=H^1(X, \O_X)$. In particular, in characteristic zero the
abelian part of  $\pi^U(X,x)$ is equal to the group of the dual vector space
$H^1 (X, \O _X)^*$ (see \cite[Chapter IV, Proposition
2]{No}).

In positive characteristic, the abelian part of $\pi ^U(X,x)$ is
the inverse limit of Cartier duals of finite local group
subschemes of $\Pic X$ (cf. \cite[Chapter IV, Proposition 6]{No}).

\section{Numerically flat sheaves on products of curves}

In this section we keep the following notation. Let $X$ and $Y$ be
complete $k$-varieties. Then $p, q$ are the projections of
$X\times Y$ onto $X$ and $Y$, respectively. Let $E$ be a coherent
sheaf on $X\times Y$. For a point $y\in Y$ we set
$E_y=p_*(E\otimes \O_{X\times \{y\}})$. Similarly,
$E_x=q_*(E\otimes \O_{\{x\} \times Y})$ for a point $x\in X$.

Let us consider the following proposition in the special case when
$X$ and $Y$ are curves.

\begin{Proposition} \label{curves}
Let $X$ and $Y$ be smooth projective curves. Let $F$ be a
locally free sheaf on $X\times Y$, such that $F_{x_1}$ is semistable for
some $x_1\in X$. Assume that $F$ is numerically trivial. Then for
any closed points $y_1, y_2\in Y$ the corresponding locally free
sheaves $F_{y_1}$ and $F_{y_2}$ are isomorphic. Moreover, the bundles
$F_x$ are semistable and S-equivalent for all closed points $x\in
X$.
\end{Proposition}

\begin{proof}
Let us fix a point $x_1\in X$. By  Faltings's theorem
\cite[Theorem I.2]{Fa} (see also \cite[Remark 3.2 (b)]{Se}) there
exists a rank $r'$ vector bundle $E$ on $Y$ such that $H^*(Y,
 F _{x_1}\otimes E)=0$. Note that this condition implies that $E$
is semistable (see \cite[Theorem 6.2]{Se}). This bundle defines a
global section $\Theta _{E}$ of a line bundle $L = \det p_{!}(F
\otimes q^* E) ^{-1}$. Set-theoretically, the zero set of this
section is equal to $\{x\in X : H^*(Y, F _{x}\otimes E)\ne 0\}$.

But by the Grothendieck--Riemann--Roch theorem (see, e.g., \cite[Appendix A,
Theorem 5.3]{Ha}) for every $u\in
K(Y)$ we have
$$\ch (p_{!} (F \cdot q^*u))=p_* (\ch (F)\cdot q^* (\ch (u)\cdot \td (X)))
=r p_*(q^*(\ch (u) \cdot \td (X)))$$ and the only non-zero part in
the last term is in degree zero. Therefore $L$ has degree $0$.
Since $L$ has the section $\Theta _{E}$ non-vanishing at $x_1$ it
follows that $L$ is trivial and $H^*(Y,  F _{x}\otimes E)=0$ for
all $x\in X$. By \cite[Theorem 6.2]{Se} this implies that $F_x$ is
semistable for every $x\in X$. To prove that $F_x$ are
S-equivalent one can use determinant line bundles on the moduli
space of semistable vector bundles on $X$. Since we do not use
this fact in the following we omit the proof. An alternative proof
can be found in \cite[Lemma 4.2]{Se}.

The rest of the proof is similar to part of proof of \cite[Theorem
I.4]{Fa} (see also Step 5 in proof of \cite[Theorem 4.2]{He}). Let
us fix a point $y_1\in Y$ and take any non-trivial extension
$$0\to E\to E'\to \O_{y_1} \to  0.$$
The Quot-scheme $\cQ$ of rank $0$ and degree $1$ quotients of $E'$
is isomorphic to $\PP (E')$. Let us set  $E_{\pi}=\ker \pi$ for a
point $[\pi: E'\to \O_{y}]\in \cQ$. The set $U$ of points
$[\pi]\in \cQ$ such that $H^*(Y,  F _{x_1} \otimes E_{\pi})= 0$ is
non-empty and open.  The same arguments as before show that
$H^*(Y,  F _{x} \otimes E_{\pi} )=0$ for all $x\in X$ and $[\pi] \in
U$. Applying $p_*$ to the sequence
$$0\to F \otimes q^*E_{\pi} \to F\otimes q^*E'\mathop{\longrightarrow}^{\id_F\otimes \pi} F\otimes q^*\O_{y}\to 0$$
for $[\pi : E'\to \O_{y}] \in U$ we see that $p_* (F\otimes
q^*E')\simeq F_y$. Therefore $F_{y_1}\simeq F_y$ for points $y$ in
some non-empty open subset of $Y$. Since similar arguments apply
to any other point $y_2\in Y$ we see that $F_{y_1}$ and $F_{y_2}$
are isomorphic for all points $y_1, y_2\in Y$.
\end{proof}

The following corollary is analogous to \cite[Proposition 2.4]{Gi}
in case of stratified sheaves:

\begin{Corollary} \label{fibers-on-product}
Let $X$ be a normal complete variety and let $Y$ be a complete
variety. Let $E$ be a numerically flat sheaf on $X\times Y$. Then
for any closed points $y_1, y_2\in Y$ the corresponding locally
free sheaves $E_{y_1}$ and $E_{y_2}$ are isomorphic. In
particular, the sheaf $q_*E$ is locally free.
\end{Corollary}

\begin{proof}
Using Chow's lemma it is easy to see that there exists an
irreducible curve on $Y$ containing both $y_1$ and $y_2$. Taking
its normalization we can  replace $Y$ by a smooth projective curve
and prove the assertion in this special case. So in the following
we assume that $Y$ is a smooth projective curve.

Now we prove the assertion assuming that $X$ is projective. The
proof is by induction on the dimension $d$ of $X$. For $d=1$ the
assertion follows from Proposition \ref{curves}. For $d\ge 2$ let
us fix a divisor $D$ on $X$ and consider the following exact
sequence
$$\Hom (E_{y_1}, E_{y_2})\mathop{\longrightarrow}^{\alpha}
\Hom (E_{y_1}|_D, E_{y_2}|_D)  \longrightarrow H^1(\cHom (E_{y_1},
E_{y_2})\otimes \O_X (-D)).$$ Taking a sufficiently ample divisor
$D$, we can assume that $D$ is smooth and $H^1(\cHom (E_{y_1},
E_{y_2})\otimes \O_X (-D))=0$ (here we use $d\ge 2$; see
\cite[Chapter III, Corollary 7.8]{Ha}). But then the map $\alpha$
is surjective and the isomorphism $E_{y_1}|_D \simeq E_{y_2}|_D$
(coming from the inductive assumption) can be lifted to a
homomorphism $\varphi: E_{y_1}\to E_{y_2}$. Since $\varphi$ is
injective at the points of $D$ and $E_{y_1}$ is torsion-free, it
follows that $\varphi$ is an injection. But then it must be an
isomorphism.

To prove the assertion in case when $X$ is non-projective we can
use Chow's lemma. Namely, there exists a normal projective variety
$\ti X$ and a birational morphism $f: \ti X\to X$. Let us set
$g=f\times \id _Y: \ti X\times Y\to X\times Y$. By the previous
part of the proof we know that $f^*(E_{y_1})=(g^*E)_{ y_1}\simeq
(g^*E)_{y_2}=f^*(E_{y_2})$. But by Zariski's main theorem we know
that $f_*\O_{\ti X}=\O_X$. Hence by the projection formula we have
$E_{y_1}\simeq f_*f^*(E_{y_1}) \simeq f_*f^*(E_{y_2}) \simeq
E_{y_2}$.

The last part of the corollary follows from Grauert's theorem (see
\cite[Chapter III, Corollary 12.9]{Ha}).
\end{proof}

\section{S-fundamental group scheme of a product}

The following result was conjectured both by the author in
\cite[Section 8]{La} and by V. Mehta in \cite[Remark 5.12]{Me}:

\begin{Theorem} \label{product}
Let $X$ and $Y$ be complete $k$-varieties. Let us fix $k$-points
$x_0\in X$ and $y_0\in Y$. Then the natural homomorphism
$$ \pi_1^S(X\times _k Y, (x_0,y_0)) \to \pi_1^S(X,x_0)\times _k\pi_1^S(Y,y_0)$$
is an isomorphism.
\end{Theorem}

\begin{proof}
In characteristic zero the assertion follows from the
corresponding fact for topological fundamental groups of complex
varieties and the Lefschetz principle. More precisely, in the case
of complex varieties the assertion follows from the corresponding
isomorphism of topological fundamental groups by passing to a
pro-unitary completion. In general, the fact follows from the
Lefschetz principle if one notes that Theorem \ref{product}
follows from Lemma \ref{push-forward} and if we use the Lefschetz
principle for this special assertion (note that to apply Lefschetz
principle we need to reformulate Theorem \ref{product} as it
involves group schemes which are not of finite type over the
field).

So in the following we can assume that the characteristic of $k$ is positive
(but we need it only to prove the next lemma which is
the main ingredient in proof of Theorem \ref{product}).

\begin{Lemma} \label{push-forward}
Let us assume that $X$ and $Y$ are normal and projective. Let $E$
be a numerically flat sheaf on $X\times Y$. Then $p_*E$ is
numerically flat.
\end{Lemma}

\begin{proof}
By Corollary \ref{fibers-on-product} $p_*E$ is locally free. Let
us fix a point $x_0\in X$. Then the sheaf $G_n=p_* ((F^n_X\times
\id _Y)^* E)=(F^n_X)^* (p_*E)$ is locally free of rank
$a=h^0(Y,E_{x_0})$.

Now let us consider the set $A$ of all numerically flat sheaves on
$X\times Y$. This set is bounded by Theorem \ref{bound}, so there
exist a scheme $S$ of finite type over $k$ and an $S$-flat sheaf
$\E$ on $X\times Y\times S$ such that the set of restrictions
$\{\E _s \}_{s\in S}$ contains all the sheaves in the set $A$. Let
us consider a subscheme $S'\subset S$ defined by $S'=\{s\in S:
h^0(Y, (\E_s)_{x_0})\ge a\}$. By semicontinuity of cohomology,
$S'$ is a closed subscheme of $S$. Let us consider an open subset
$U\subset S'$ that corresponds to points $s\in S'$ where
$h^0(Y,(\E_s)_{x_0})= a$. We consider $U$ with the reduced scheme
structure. By abuse of notation, the restriction of $\E$ to $U$
will be again denoted by $\E$.

 We claim that the set $\{p_*(\E_s)\} _{s\in U}$ is a
bounded set of sheaves. To prove this let us consider the
following diagram
$$ \xymatrix{& Y\ar@(ur,ul)[rr]^{j_{(x,s)}}\ar[r]^{j_{x}}\ar[d]^{p_{x}}
& X\times Y\ar[r]^{\tilde j_s} \ar[d]^{p} &X \times Y\times U
\ar[d]^{\tilde p}\\
&\Spec k\ar@(dr,dl)[rr]^{i_{(x,s)}}\ar[r]^{i_{x}}&X\ar[r]^{\tilde
i_s} &X\times U\\ }$$ in which the vertical maps are canonical
projections and the horizontal maps are embeddings corresponding to
fixed points $x\in X$ and $s\in U$. Let us recall that $p_*(\tilde
j_s^* \E)$ is locally free by Corollary \ref{fibers-on-product}
and the definition of $U$.
Moreover, $\tilde p_* \E$ is locally free as $h^0(Y,(\E_s)_{x})=
h^0(Y,(\E_s)_{x_0})= a$ for every $x\in X$ by Corollary
\ref{fibers-on-product}. Therefore the above diagram induces the
commutative diagram
$$\xymatrix{
& i_x^*p_*(\tilde j_s^* \E) \ar[r] &(p_x)_*j_x^*(\tilde j_s^* \E)\\
& i_{(x,s)}^*(\tilde p_*\E) \ar[r]\ar[u]&
(p_x)_*j^*_{(x,s)}\E\ar[u]^{\simeq}}$$ in which the horizontal
maps are isomorphisms by Grauert's theorem (see \cite[Chapter III,
Corollary 12.9]{Ha}). Hence the remaining vertical map is also an
isomorphism. This shows that $(\tilde p_*\E)_s\simeq p_* (\E_s)$,
which gives the required claim.

By our assumptions there exists a sequence $(s_n)_{n\in \NN}$ of
points of $U$ such that $p_*(\E_{s_n})\simeq G_n$. Therefore the
set $\{G_n\} _{n\in \NN}$ is bounded.

Now let us take any morphism $f:C\to X$ from a smooth projective
curve $C$. Since $(F^n_C)^* (f^*p_*E)\simeq f^*(G_n)$ and the set
$\{f^*(G_n)\} _{n\in \NN}$ is bounded, we see that $f^*p_*E$ is
semistable of degree $0$. More precisely, semistability follows
from the fact that sequences of slopes of maximal destabilizing
subsheaves and minimal destabilizing quotients of Frobenius
pull-backs of $f^*p_*E$ are bounded. Similarly, the sheaf
$f^*p_*E$ has degree $0$ since its Frobenius pull-backs have
bounded degree. Therefore $p_*E$ is numerically flat.
\end{proof}

\medskip

Now we can go back to the proof of Theorem \ref{product}.

The first part of proof is the same as the analogous part of proof
of \cite[Chapter IV, Lemma 8]{No}. Namely, the homomorphism $
\pi_1^S(X\times _k Y, (x_0,y_0)) \to \pi_1^S(X,x_0)\times
_k\pi_1^S(Y,y_0)$ is induced by projections $p: X\times Y\to X$
and $q: X \times Y\to Y$. Let $i:X\to X\times Y$ be the embedding
onto $X\times \{y_0\}$ and let $j:Y\to X\times Y$ be the embedding
onto $\{x_0\} \times Y$. Since $p\, i=\id _X$ and $q\, i$ is
constant, the composition $ \pi_1^S(X,x_0)\to \pi_1^S(X\times _k
Y, (x_0,y_0)) \to \pi_1^S(X,x_0)\times _k\pi_1^S(Y,y_0)$ is an
embedding onto the first component. Similarly $q\, j=\id _Y$ is
the embedding onto the second component, so the homomorphism $
\pi_1^S(X\times _k Y, (x_0,y_0)) \to \pi_1^S(X,x_0)\times
_k\pi_1^S(Y,y_0)$ can be split and in particular it is faithfully
flat.

Hence we only need to prove that it is a closed immersion.

Let us first assume that $X$ and $Y$ are normal and projective.
Note that the sheaf $F=\cHom(q^*E_{x_0}, E)$ is numerically flat.
By our assumptions and Lemma \ref{push-forward} the sheaf $p_*F$
is numerically flat. The induced map
$$p^*p_* F \otimes q^* E_{x_0}\to E$$
is surjective as its restriction to $\{x_0\} \times Y$ corresponds
to the surjective map $\Hom (E_{x_0}, E_{x_0})\otimes E_{x_0}\to
E_{x_0}$. Now \cite[Proposition 2.21 (b)]{DM} implies that the
natural homomorphism $ \pi_1^S(X\times _k Y, (x_0,y_0)) \to
\pi_1^S(X,x_0)\times _k\pi_1^S(Y,y_0)$ is a closed immersion and
therefore it is an isomorphism.

\medskip

Now we need the following lemma:

\begin{Lemma} \label{push-forward3}
Let $X$ and $Y$ be complete $k$-varieties such that the natural map
$\pi_1^S(X\times _k Y, (x_0,y_0)) \to \pi_1^S(X,x_0)\times _k\pi_1^S(Y,y_0)$
is an isomorphism. Let $E$ be a numerically flat sheaf on $X\times Y$.
Then for every integer $i$ the sheaf $R^ip_*E$ is numerically flat and
cohomology and base change commute for $E$ in all degrees.
\end{Lemma}

\begin{proof}
Since $\pi_1^S(X\times _kY, (x_0,y_0))\simeq \pi_1^S(X,x_0)\times
_k\pi_1^S(Y,y_0)$, $E$ is a subsheaf of a sheaf $p^*E_0\otimes
q^*F_0$ for some numerically flat sheaves $E_0$ and $F_0$. This is
an easy fact from representation theory as every $G_1\times
G_2$-module is a submodule of the tensor product of $G_1$ and
$G_2$-modules. Alternatively, \cite[Proposition 2.21 (b)]{DM} implies that
every numerically flat sheaf on $X\times Y$ is a quotient of a sheaf of
the form $p^*E_0\otimes q^*F_0$, which implies the required statement by
taking the duals.

The quotient $(p^*E_0\otimes q^*F_0)/E$ is also numerically flat, so it is
a subsheaf of a sheaf of the form $p^*E_1\otimes q^*F_1$ for some numerically
flat sheaves $E_1$ and $F_1$. Inductively we can therefore construct the following
acyclic complex of sheaves on $X\times Y$:
$$
0\to E\to p^*E_0\otimes q^*F_0 \to  p^*E_1\otimes q^*F_1\to ...\to p^*E_i\otimes q^*F_i\to ... \leqno{(*)}
$$
Let us set $\C ^i=p^*E_i\otimes q^*F_i$. Note that
$R^ip_*\C^j\simeq E_j ^{\oplus h^i(Y,F_j)}$ is numerically flat and consider the
following spectral sequence
$$E^{ij}_1=R^ip_*\C^j \Longrightarrow R^{i+j}p_*\C^{\bullet}\simeq R^{i+j}p_*E.$$
Since the category of numerically flat sheaves on $X$ is abelian,
kernels and cokernels of objects from this category are also
numerically flat. This implies that the limit $R^{i+j}p_*E$ of the above spectral sequence
is also numerically flat.

Now note that the complex $(*)$ restricted to $\{x\} \times Y$ remains acyclic, as all the sheaves
in this complex are locally free. Therefore we have a commutative diagram
of spectral sequences
$$\xymatrix{
& R^ip_*\C^j\otimes k(x)
\ar[d]& \Longrightarrow &R^{i+j}p_*E \otimes k(x)
\ar[d]\\
& H^i (Y, \C ^j_x) &\Longrightarrow &  H^{i+j}(Y, E_x) .} $$
Since $\C ^i=p^*E_i\otimes q^*F_i$, the left vertical map in this diagram is an isomorphism. Therefore
the right vertical map is also an isomorphism. But this implies that
cohomology and base change commute for $E$ in all degrees (see \ref{cohomology_and_base-change}).
\end{proof}

\medskip

Now let us return to the proof of Theorem \ref{product}. Let us
first assume that $X$ is normal and projective. By Chow's lemma
there exists a projective variety $\ti Y$ and a birational
morphism $f: \ti Y\to Y$. Passing to the normalization, we can
assume that $\ti Y $ is normal. Consider the base change diagram
$$\xymatrix{
& X\times \tilde Y \ar[r]^g \ar[d]^{\tilde q}& X\times Y\ar[d]^q\\
& \tilde Y \ar[r]^f&  Y}.$$ Let us take two closed points $y_1,
y_2\in Y$. Let us choose closed points $\tilde y_1, \tilde y_2\in
\tilde Y$ mapping onto $y_1, y_2$, respectively. Then
$h^0((g^*E)_{\tilde y_1})=h^0((g^*E)_{\tilde y_2})$ and hence
$h^0(E_{y_1})=h^0( E_{y_2})$. By Grauert's theorem (and \ref{cohomology_and_base-change})
this implies that $q_*E$ is locally free and cohomology and base change commute for $E$
in all degrees. In particular, we have and $f^* (q_*E)\simeq {\ti q}_* (g^*E)$.
But ${\ti q}_* (g^*E)$ is numerically flat, so $q_*E$ is also numerically flat.

In this case the same proof as in the previous case shows that $ \pi_1^S(X\times
_k Y, (x_0,y_0)) \to \pi_1^S(X,x_0)\times _k\pi_1^S(Y,y_0)$ is an
isomorphism.

Now we can again apply Chow's lemma to prove that if $X$ and $Y$
are complete and $E$ is numerically flat on $X\times Y$ then
$p_*E$ is numerically flat. As in the  previous case this implies
that $ \pi_1^S(X\times _k Y, (x_0,y_0)) \to \pi_1^S(X,x_0)\times
_k\pi_1^S(Y,y_0)$ is a closed immersion.
\end{proof}

\medskip
\begin{Remark}
Most of the proof of Theorem \ref{product} works in an arbitrary
characteristic. But the proof of Lemma \ref{push-forward} uses
positive characteristic. The characteristic zero version of
Theorem \ref{product} would follow from the positive
characteristic case if one knew that the reduction of a semistable
complex bundle for some characteristic is strongly semistable
(this is a weak version of {Miyaoka's problem}; see \cite[Problem
5.4]{Miy}). This problem seems to be open even for numerically
flat bundles on a product of two curves of genera $\ge 2$.
\end{Remark}

\medskip

Applying \cite[Lemma 6.3]{La} as a corollary to Theorem
\ref{product} we obtain the following result of Mehta and
Subramanian (conjectured earlier by Nori in \cite{No}).

\begin{Corollary} {\emph{(see \cite[Theorem 2.3]{MS})}}
Let $X$ and $Y$ be complete $k$-varieties. Let us fix $k$-points
$x_0\in X$ and $y_0\in Y$. Then the natural homomorphism
$$ \pi_1^N(X\times _kY, (x_0,y_0)) \to \pi_1^N(X,x_0)\times _k\pi_1^N(Y,y_0)$$
is an isomorphism.
\end{Corollary}

Note that Theorem \ref{product} and Lemma \ref{push-forward3} imply the following corollary.

\begin{Corollary} \label{push-forward2}
Let $X$ and $Y$ be complete $k$-varieties and let $E$ be a
numerically flat sheaf on $X\times Y$. Then for every integer $i$ the sheaf
$R^ip_*E$ is numerically flat and cohomology and base change commute for
$E$ in all degrees.
\end{Corollary}

\begin{Corollary} \label{Pic-product}
Let $X$ and $Y$ be projective $k$-varieties. Then
$$({\Pic} ^{\tau} (X\times Y))_{red}\simeq ({\Pic} ^{\tau} X)_{red}\times ({\Pic} ^{\tau} Y)_{red}$$
and
$$({\Pic} ^{0} (X\times Y))_{red}\simeq ({\Pic} ^{0} X)_{red}\times ({\Pic} ^{0} Y)_{red}.$$
\end{Corollary}

\begin{proof}
The category of representations of the S-fundamental group scheme
$\pi ^S_1(X,x)$ is equivalent to the category $\Vect (X)$ of
numerically flat sheaves. Since a line bundle is numerically flat
if and only if it is numerically trivial,  the group  $(\Pic
^{\tau} (X))_{red}$ is the group of characters of $\pi ^1 (X,x)$
(here we use projectivity of $X$). Now the first isomorphism
follows directly from Theorem \ref{product}.

The second isomorphism follows immediately from the first one
(if $X$ and $Y$ are smooth projective varieties one can also
give another proof using comparison of dimensions of
$\Pic ^0 (X\times Y)$ and $\Pic ^0X \times \Pic^0Y$).
\end{proof}

\section{The abelian part of the S-fundamental group scheme}

We say that a numerically flat sheaf is \emph{irreducible} if it
does not contain any proper numerically flat subsheaves (or,
equivalently, if it corresponds to an irreducible representation
of $\pi ^S_1(X,x)$). In case of projective varieties a numerically
flat sheaf is irreducible if and only if it is slope stable (with
respect to some fixed polarization; or, equivalently, with respect
to all polarizations).

If $E$ is an irreducible numerically flat sheaf then it is simple.
This follows from the fact that any endomorphism of such $E$ is
either $0$ or an isomorphism (otherwise the image would give a
proper numerically flat subsheaf).

\begin{Theorem} \label{irred-product}
Let $E$ be an irreducible numerically flat sheaf on a product
$X\times Y $ of complete varieties $X$ and $Y$. Then there exist
irreducible numerically flat sheaves $E_1$ on $X$ and $E_2$ on $Y$
such that $E\simeq p^*E_1\otimes q^*E_2$.
\end{Theorem}

\begin{proof}
Let us fix a point $x_0\in X$ and an irreducible numerically flat
subsheaf $K\subset E_{x_0}$. Let us set $F=\cHom(p^*K, E)$. Since
the sheaf $F$ is numerically flat, by Corollary
\ref{push-forward2} we know that $q_* F$ is numerically flat. It
is easy to see that the induced map
$$\varphi: q^*q_* F \otimes p^* K\to E $$
is injective on $\{ x_0\}\times Y$. Since $(q_*F)\otimes
k(x_0)=\Hom (K, E_{x_0})$ is non-zero, $q^*q_* F \otimes p^* K$ is
numerically flat and $E$ is irreducible, $\varphi$ is an
isomorphism.
\end{proof}

\medskip

The following lemma follows from the proof of Theorem
\ref{irred-product} but we give a slightly different proof without
using Theorem \ref{product} (so in particular the proof is
completely algebraic) in a case sufficient for the applications in the
next section.

\begin{Lemma}
Let  $X$ and $Y$ be  complete varieties.  Let $E$ be a numerically
flat sheaf on $X\times _k Y$. Assume that for some point $y_0\in
Y$, $E_{y_0}$ is simple (i.e., $\End (E_{y_0})=k$). Then there
exists a numerically trivial line bundle $L$ on $Y$ such that
$E\simeq p^*E_{y_0}\otimes q^*L$.
\end{Lemma}

\begin{proof}
For simplicity let us assume that $X$ is smooth and projective.
Let us set $F=\cHom(p^*E_{y_0}, E)$. Since the sheaf $F$ is
numerically flat, by Corollary \ref{fibers-on-product} we know
that $q_* F$ is locally free. Therefore the induced map
$$q^*q_* F \otimes p^* E_{y_0}\to E $$
is surjective on all fibers $\{ x\}\times Y$ and hence it is
surjective. Since $E_{y_0}$ is simple, $L=q_*F$ is a line bundle.
Therefore the above map is a surjective map of locally free
sheaves of the same rank and hence it is an isomorphism. This also
shows that $L$ is numerically trivial.
\end{proof}

\begin{Proposition} \label{simple}
Let  $X$ and $Y$ be  complete varieties.  Let $E$ be a numerically
flat sheaf on $X\times _k Y$. Assume that for some point $x_0\in
X$ and $y_0\in Y$, $E_{x_0}$ and $E_{y_0}$ are simple. Then both
$E_{x_0}$ and $E_{y_0}$ are line bundles and $E\simeq
p^*E_{y_0}\otimes q^*E_{x_0}$.
\end{Proposition}

\begin{proof}
By the above lemma we know that $E\simeq p^*E_{y_0}\otimes q^*L$
for some line bundle $L$. Therefore $E_{x_0}\simeq
q_*((p^*E_{y_0}\otimes q^*L)\otimes p^*\O_{x_0})\simeq L^{\oplus
\rk E_{y_0}}$. Since $E_{x_0}$ is simple, $E_{y_0}$ is a line
bundle and $E_{x_0}\simeq L$.
\end{proof}

\medskip

The following definition is an analogue of \cite[Definition
2.5]{Gi} where the corresponding notion is defined for stratified
sheaves.

\begin{Definition} \label{def-abelian}
Let $E$ be a numerically flat sheaf on a complete variety $X/k$.
We say that $E$ is \emph{abelian} if there exists a numerically
flat sheaf $E'$ on $X\times _k X$ and a closed point $x_0\in X$
such that $E'$ restricted to both $p^{-1}(x_0)$ and $q^{-1}(x_0)$
is isomorphic to $E$.
\end{Definition}

Proposition \ref{simple} implies that a simple abelian numerically
flat sheaf on a complete variety has rank one. Since every
numerically flat sheaf  has a filtration with irreducible
quotients, we get the following corollary describing the category
of representations of $\pi ^S_{ab} (X,x)$.

\begin{Corollary} \label{abelian}
Let $E$ be an abelian numerically flat sheaf on a complete variety
$X$. Then $E$ has a filtration $0=E_0\subset E_1\subset \dots
\subset E_r=E$ in which all quotients $E_i/E_{i-1}$ are
numerically trivial line bundles. Moreover, if $E$ is simple then
it is a line bundle.
\end{Corollary}

Note that the last part of the corollary is non-trivial as even on
complex projective varieties, a simple, semistable bundle need not
be stable.

\begin{Remark}
In the next section we will see that every numerically flat sheaf
on an abelian variety is abelian. Therefore the above corollary
generalizes \cite[Theorem 2]{MN} which describes numerically flat
sheaves on abelian varieties assuming boundedness of semistable
sheaves with fixed numerical invariants (this is proven in
\cite{La1}). Mehta and Nori proved this theorem using reduction to
the case of abelian varieties defined over an algebraic closure of
a finite field. Our proof is completely different.
\end{Remark}

\medskip

\begin{Lemma} \label{group}
Let $G$ be an affine $k$-group scheme. Then the category of
representations of the largest abelian quotient $G_{ab}$ of $G$ is
isomorphic to the full subcategory of the category of
representation of $G$, whose objects are those $G$-modules $V$ for
which there exists a $G\times G$-module $W$ such that we have
isomorphisms of $G$-modules $W|_{G\times \{e\}}\simeq W|_{\{e\}
\times G}\simeq V$.
\end{Lemma}

\begin{proof}
We give a naive proof in the case of algebraic groups. The reader is
asked to rewrite it in terms of group schemes using a precise
definition of derived subgroup scheme (see \cite[10.1]{Wa}).

Let $V$ be a $G_{ab}$-module. Since $G_{ab}$ is abelian, the
multiplication map $G_{ab}\times G_{ab}\to G_{ab}$ is a
homomorphism of affine group schemes. Hence we can treat $V$ as a
$G\times G$-module obtaining the required module $W$.

Now let $V$ be a $G$-module for which there exists  $\rho:
G\times G \to \GL (W)$ and isomorphisms of $G$-modules
$W|_{G\times \{e\}}\simeq W|_{\{e\} \times G}\simeq V$. The
representation $\tilde \rho: G\to \GL (V)$ corresponding to $V$ is
given by $\tilde \rho (g)=\rho (g,e)=\rho (e,g)$. Since
$$\rho (g,h)=\rho (g,e)\cdot \rho (e,h)=\rho (e,h)\cdot \rho (g,e),$$
we have $\tilde \rho (ghg^{-1}h^{-1})=e$. Therefore $\tilde \rho $
vanishes on the derived subgroup of $G$ and hence it defines the
required representation of $G_{ab}$.
\end{proof}

\medskip

 The largest abelian quotient of the S-fundamental group scheme
of $X$ is called the \emph{abelian part} of $\pi ^S(X,x)$ and
denoted by $\pi ^S_{ab}(X,x)$.

Lemma \ref{group}, together with Theorem \ref{product}, implies
that $\pi ^S_{ab}(X,x)$ is Tannaka dual to the (neutral Tannakian)
category of abelian numerically flat sheaves on $X$. This also
explains the name ``abelian'' in Definition \ref{def-abelian}.

\begin{Remark}
Note that Corollary \ref{abelian} follows immediately from the
interpretation of abelian numerically flat bundles given above and
the fact that irreducible representations of abelian group schemes
are one-dimensional (see \cite[Theorem 9.4]{Wa}). In fact, this
also proves that all subsheaves $E_i$ in the filtration are
abelian.
\end{Remark}

\medskip

Let $G$ be an abelian group. Then we define the $k$-group scheme
$\Diag (G)$ by setting $$(\Diag (G)) (\Spec A)={\Hom} _{\rm gr}
(G, A^{\times})$$ for a $k$-algebra $A$, where ${\Hom} _{\rm gr}
(G, A^{\times})$ denotes all group homomorphisms from $G$ to the
group of units in $A$. This is the same as taking $\Spec k[G]$
with the natural $k$-group scheme structure.

\begin{Theorem} \label{abelian-part}
Let $X$ be a smooth projective variety defined over an
algebraically closed field $k$. If the characteristic of $k$ is
positive then we have an isomorphism
$$ \pi ^S_{ab}(X, x)\simeq \lim_{{\leftarrow}}{\hat{G}} \times \Diag (({\Pic} ^{\tau}X)_{red}),$$
where $\hat{G}$ denotes the Cartier dual of $G$ and the inverse
limit is taken over all finite local group subschemes $G$ of $\Pic
^0 X$. If $k$ has characteristic zero then we have
$$\pi ^S_{ab}(X, x)\simeq H^1(X, \O_X)^*\times  \Diag
({\Pic} ^{\tau}X) .$$
\end{Theorem}

\begin{proof}
By \cite[Theorem 9.5]{Wa}  $ \pi ^S_{ab}(X, x)$ is a product of its unipotent
and multiplicative parts.

By Nori's results mentioned in Subsection \ref{unipotent}
(see \cite[Chapter IV, Proposition 6]{No}) we know
that in positive characteristic the abelian part of the unipotent
part  (or equivalently, the unipotent part of the abelian part)
$\pi ^U (X, x)$ of the S-fundamental group is isomorphic to the
inverse limit of Cartier duals of finite local group subschemes of
$\Pic X$. In characteristic zero, $\pi ^U _{ab}(X, x)=H^1(X,
\O_X)^*$.

 On the other hand, since the character group
of $\pi ^S(X, x)$ is isomorphic to the reduced scheme underlying
$\Pic ^{\tau} X$ (in characteristic zero $\Pic ^\tau X$ is already
reduced), the diagonal part of $\pi ^S_{ab}(X, x)$ is given by
$\Diag (({\Pic} ^{\tau}X)_{red})$, which finishes the proof.
\end{proof}

\medskip

Let us recall that the Neron--Severi group  $\NS (X)=({\Pic}
X)_{red}/(\Pic ^0X)_{red}$ is finitely generated. We have a short
exact sequence
$$0\to ({\Pic}^0X)_{red} \to ({\Pic} ^{\tau}X)_{red}\to \NS (X)_{tors} \to 0,$$
where $ \NS (X)_{tors}$ is the torsion group of $\NS (X)$ (it is a
finite group). Therefore we have a short exact sequence
$$0\to \Diag ( \NS (X)_{tors})\to \Diag (({\Pic} ^{\tau}X)_{red}) \to \Diag (({\Pic}^0X)_{red})\to 0.$$

\section{Numerically flat sheaves on abelian varieties}

Let $A$ be an abelian variety defined over an algebraically closed
field of characteristic $p$ and let $A_n$ be the kernel of the
multiplication by $n$ map $n_A: A\to A$.

Let $A_{p^n}^r$ be the reduced part of $A_{p^n}$. Then $A_{p^n}$
is a product of $A_{p^n}^r$, its Cartier dual $\hat{A}_{p^n}^r$
(which is a local and diagonalizable group scheme) and a
local--local group scheme $A_{p^n}^0$ (see \cite[Section 15]{Mu};
\emph{local--local} means that both the group scheme and its
Cartier dual are local). Let $r$ be the $p$-rank of $A$. Then
$A_{p^n}^r\simeq (\ZZ/p^n\ZZ)^r$, $\hat{A}_{p^n}^r\simeq
(\mu_{p^n})^r$ and $A_{p^n}^0$ is a group scheme of order
$p^{2(\dim A -r)n}$.

The \emph{$p$-adic discrete Tate group} $T_p^d(A)$ is defined as
the inverse limit $\displaystyle\lim_{\longleftarrow} A_{p^n}^r$
(see \cite[Section 18]{Mu}). This is a $\ZZ _p$-module. Similarly,
we define \emph{the $p$-adic local--local Tate group} $T_p^0(A)$
as the inverse limit  $\displaystyle\lim_{{\longleftarrow}}
A_{p^n}^0$. For an arbitrary prime $l$ (possibly $l=p$) we also
define the \emph{$l$-adic Tate group scheme} as
$T_l(A)=\displaystyle\lim_{{\longleftarrow}}A_{l^n}$ (note that
for $l=p$ our notation differs from that in \cite{Mu}).

The following theorem is analogous to \cite[Theorem 21]{dS}:

\begin{Theorem} \label{abelian-variety}
Let $A$ be an abelian variety defined over an algebarically closed
field $k$. Then $\pi ^S(A, 0)$ is abelian and it decomposes as a
product of its unipotent and diagonal parts. Moreover, if the
characteristic of $k$ is positive then
$$\pi ^S(A, 0)\simeq {T}_p^0 (A)\times T_p^d(A)\times \Diag
({\Pic} ^0A) .$$ In characteristic zero we have
$$\pi ^S(A, 0)\simeq H^1(A, \O_A)^*\times \Diag
({\Pic} ^0A).$$
\end{Theorem}

\begin{proof}
Let us first remark that every numerically flat bundle $E$ on $A$
is abelian. To show this one can take the addition map $m: A\times
_k A\to A$. Then $E'=m^*E$ restricted to either $p^{-1}(0)$ or
$q^{-1} (0)$ is isomorphic to $E$. This shows that $\pi ^S(A, 0)$
is abelian and we can use Theorem \ref{abelian-part}. In positive
characteristic, local group subschemes of the dual abelian variety
$\hat A=\Pic ^0A$ are of the form $ A_{p^n}^0\times
\hat{A}_{p^n}^r$. So the inverse limit of their Cartier duals is
isomorphic to ${T}_p^0 (A)\times T_p^d(A)$. On the other hand,
since $\Pic ^{\tau} A=\Pic ^0A$, the diagonal part of $\pi ^S(A,
0)$ is given by $\Diag ({\Pic} ^0A)$.
\end{proof}

\medskip

We can also give another proof of the above theorem without
using that $\pi ^S(A, 0)$ is abelian (which uses the rather
difficult Theorem \ref{product}). Namely, as in the proof of \cite[Theorem
21]{dS}  it is sufficient to show that for every indecomposable
numerically flat sheaf $E$ on $A$ there exists a unique line
bundle $L\in \Pic ^0(A)$, such that $L \otimes E$ has a filtration
by subbundles with each successive quotient trivial. But by
Corollary \ref{abelian} every numerically flat bundle $E$ on $A$
has a filtration $0=E_0\subset E_1\subset \dots \subset E_r=E$ in
which all quotients $E_i/E_{i-1}$ are numerically trivial line
bundles (in fact, this filtration is just a Jordan--H\"older
filtration for an arbitrary polarization, but Corollary
\ref{abelian} says a bit  more about this filtration).

Now if $E$ is indecomposable then it is easy to see that all the
quotients in the filtration are isomorphic (otherwise one can
prove by induction that the filtration would split as $H^1(A,L)=0$
for a non-trivial line bundle $L$ on the abelian variety $A$),
which finishes the proof.

\medskip

\begin{Corollary}
Let $A$ be an abelian variety defined over an algebraically closed
field of positive characteristic. Then
$$ \pi ^N (A, 0)\simeq \lim _{\longleftarrow}A_n\simeq
\prod _{l \, {\rm prime}} T_l(A).$$
\end{Corollary}

The above corollary also follows from \cite[Remark 3]{No2} and
\cite[Theorem 2.3]{MS}. It is well known that the above corollary
implies the Serre--Lang theorem (see \cite[Section 18]{Mu}),
although this sort of proof is much more complicated than the original
one. In \cite{No2} Nori had to use the Serre--Lang theorem to
prove the corollary.

\section{The Albanese morphism}

Let $X$ be a smooth projective variety and let $x\in X$ be a fixed
point. Let $\alb _X:X\to \Alb X$ be the Albanese morphism mapping
$x$ to $0$. The variety $\Alb X$ is abelian and it is dual to the
reduced scheme underlying $\Pic ^0X$. Since the S-fundamental
group scheme of an abelian variety is abelian, we have the induced
homomorphism $\pi^S_{ab}(X,x)\to \pi^S_1(\Alb X, 0)$. The
following proposition proves that this homomorphism is faithfully
flat and it describes its kernel.

\begin{Theorem} \label{Albanese}
We have the following short exact sequence:
$$0\to \lim_{G\subset \Pic^0 X} \widehat{{G/G_{red}}}\times \Diag (\NS (X)_{tors})\to \pi^S_{ab}(X,x)\to
\pi^S_1(\Alb X, 0)\to 0,$$ where the limit is taken over all local
group schemes $G\subset \Pic^0 X$.
\end{Theorem}

\begin{proof}
The Albanese morphism induces $\alb_X^*:\Pic ^0(\Alb X)\simeq(\Pic
^0X)_{red} \to \Pic^0X$, which is the natural closed embedding.
 By Theorem \ref{abelian-part} we have the
commutative diagram
$$ \xymatrix{
&\pi ^S_{ab} (X,x)\ar[r]\ar[d]^{\simeq} &\pi^S_1(\Alb X, 0)\ar[d]^{\simeq} \\
& {\displaystyle\lim_{G\subset \Pic ^0X}} \hat{G}\times \Diag
(({\Pic}^{\tau} X)_{red})\ar[r]& {\displaystyle\lim_{G\subset
(\Pic ^0X)_{red}}}\hat{G}\times \Diag (({\Pic}^0 X)_{red}), }$$
where the limits are taken over local group schemes $G$ and the
lower horizontal map is induced by $\alb_X^*$. In particular, the
homomorphism $\pi^S_{ab}(X,x)\to \pi^S_1(\Alb X, 0)$ is faithfully
flat and one can  easily describe its kernel.
\end{proof}

\medskip

Theorem \ref{Albanese} implies in particular that if $\Pic^{0}X$
is reduced then the sequence
$$0\to \Diag (\NS (X)_{tors})\to \pi^S_{ab}(X,x)\to \pi^S_1(\Alb X, 0)\to
0$$ is exact. Moreover, if $\Pic^{\tau}X$ is connected and reduced
(e.g., if $X$ is a curve or  a product of two curves as in proof
of Proposition \ref{curves}) then $\pi^S_{ab}(X,x)\to \pi^S_1(\Alb
X, 0)$ is an isomorphism.

\begin{Corollary} {\emph{(cf. \cite[Chapter III, Corollary
4.19]{Mi})}} We have
$$0\to \lim_{G\subset \Pic^0 X} \widehat{{G/G_{red}}}\times \Diag (\NS (X)_{tors})\to \pi^N_{ab}(X,x)\to
\pi^N_1(\Alb X, 0)\to 0$$ and
$$0\to \left(\lim_{G\subset \Pic^0 X} \widehat{{G/G_{red}}}\right)_{red}\times\Diag (\NS (X)_{tors})\to \pi^{\et}_{ab}(X,x)\to
\pi^{\et}_1(\Alb X, 0)\to 0.$$
\end{Corollary}

\begin{proof}
For a group scheme $H$ let us consider the directed system of
quotients $H\to G$, where $G$ is a finite group scheme (or an
\'etale finite group). Let us consider the functor $F$ ($F'$) from
the category of affine group schemes to the category of pro-finite
group schemes (pro-finite groups, respectively), which to $H$
associates the inverse limit of the above directed system. Note
that these functors are exact, as the directed systems that we
consider satisfy the Mittag--Leffler condition. Moreover, the
functor $F$ ($F'$) applied to $\pi ^S_1(X,x)$ gives $\pi^N_1(X,x)$
($\pi^{\et}_1(X,x)$, respectively). Therefore the required
assertions follow by applying the functors $F, F'$ to the sequence
from Theorem \ref{Albanese}.
\end{proof}

\section{Varieties with trivial \'etale fundamental group}

This section contains computation of the S-fundamental group
scheme for varieties with trivial \'etale fundamental group. It
contains a generalization of the main result of \cite{EM2} and it
is based on  the method of \cite{EM1}.

\medskip

Let $X$ be a smooth projective variety defined over an
algebraically closed field of characteristic $p>0$. Let $M_r$  be
the moduli space of rank $r$ slope stable bundles with numerically
trivial Chern classes. It is known that it is a quasi-projective
scheme.

A closed point $[E]\in M_r$ is called \emph{torsion} if there
exists a positive integer $n$ such that $(F_X^n)^*E\simeq E$.

\begin{Lemma}
Assume that $\pi^{\et} (X,x)=0$. If $E$ is a strongly stable
numerically flat vector bundle then its rank is equal to $1$ and
there exists $n$ such that $(F_X^n)^*E\simeq \O_X$.
\end{Lemma}

\begin{proof}
By assumption all vector bundles $E_n=(F_X^n)^*E$ are stable. Let
$N$ be the Zariski closure of the set $\{[E_0], [E_1], \dots,  \}$
in $M_r$. If $N$ has dimension $0$ then some Frobenius pull back
$E'=(F_X^n)^*E$ is a torsion point of $M=M_r$. In this case by the
Lange--Stuhler theorem (see \cite[Satz 1.4]{LS}) there exists a finite \'etale covering
$f:Y\to X$ such that $f^*E'$ is trivial. But by assumption there
are no non-trivial \'etale coverings so $E'$ is trivial.

Therefore we can assume that $N$ has dimension at least $1$. Note
that the set $N'$ of irreducible components of $N$ of dimension
$\ge 1$ is Verschiebung divisible (see \cite[Definition
3.6]{EM1}), since $V|_N$ is defined at points $E_n$ for $n\ge 1$.
Now we can proceed exactly as in proof of \cite[Theorem 3.15]{EM1}
to conclude that the trivial bundle is dense in $N'$, a
contradiction.
\end{proof}

\begin{Proposition}
Assume that $\pi^{\et} (X,x)=0$. Let $E$ be a rank $r$ numerically
flat vector bundle on $X$. Then there exists some integer $n\ge 0$
such that $(F_X^n)^*E\simeq \O_X^r$.
\end{Proposition}

\begin{proof}
Proof is by induction on the rank $r$ of $E$. When $r=1$ then the
assertion follows from the above lemma.

Let us recall that there exists an integer $n$ such that
$(F_X^n)^*E$ has a Jordan H\"older filtration $E_0=0\subset
E_1\subset \dots \subset E_m=(F_X^n)^*E$ in which all quotients
$E^i=E_i/E_{i-1}$ are strongly stable numerically flat vector
bundles (see \cite[Theorem 4.1]{La}). By taking further Frobenius
pull backs and using the above lemma we can also assume that
$E^i\simeq \O_X$.

By our induction assumption taking further Frobenius pull backs we
can also assume that $E_m/E_1\simeq \O_X^{r-1}$. Now we need to
show that there exists some integer $s$ such that the extension
$$0\to (F^s_X)^*E_1\to (F^s_X)^*E\to (F^s_X)^*(E_m/E_1)\to 0$$
splits. To prove that it is sufficient to note that the
endomorphism $F^*: H^1(X,\O_X)\to H^1(X,\O_X)$ is nilpotent. But
we know that $F^*$ induces the Fitting decomposition $H^1(X,
\O_X)= H^1(X, \O_X)_{s} \oplus H^1(X, \O_X)_{n}$ into stable and
nilpotent parts and the assertion follows from equality
$$H^1(X, \O_X)_{s}=\Hom (\pi_1^{\et}(X,x), \ZZ/p)\otimes _{\FF _p}k=0.$$
\end{proof}

\begin{Corollary}
If $\pi^{\et} (X,x)=0$ then $\pi^{S} (X,x)\simeq \pi^{N} (X,x)$.
\end{Corollary}

As a special case we get \cite[Theorem 1.2]{EM2}: if $\pi^{N}
(X,x)=0$ then $\pi^{S}(X,x)=0$.

\bigskip

{\bf Acknowledgements.}

The author was partially supported by a Polish MNiSW grant
(contract number NN201265333). The author would like to thank the
referee for the remarks that helped to improve the paper.

\end{document}